\documentclass[12pt,a4paper,twoside]{amsart}
\usepackage{amsmath}
\usepackage{amssymb}
\usepackage{amsfonts}
\usepackage{a4}

\newtheorem{theorem}{Theorem}[section]
\newtheorem{lemma}[theorem]{Lemma}
\newtheorem{fact}[theorem]{Fact}
\newtheorem{proposition}[theorem]{Proposition}
\newtheorem{corollary}[theorem]{Corollary}
\theoremstyle{definition}
\newtheorem{definition}[theorem]{Definition}
\newtheorem{claim}[theorem]{Claim}
\newtheorem{example}[theorem]{Example}
\newtheorem{question}[theorem]{Question}
\newtheorem{conjecture}[theorem]{Conjecture}
\newtheorem{remark}[theorem]{Remark}

\newcommand{\mC}{{\mathbb C}}

\newcommand{\mE}{{\mathbb E}}
\newcommand{\mF}{\mathbb F}

\newcommand{\mN}{\mathbb N}

\newcommand{\mQ}{\mathbb Q}
\newcommand{\mR}{{\mathbb R}}

\newcommand{\mA}{\mathbb A}
\newcommand{\mZ}{{\mathbb Z}}

\newcommand{\bA}{{\mathbf A}}
\newcommand{\bX}{{\mathbf X}}
\newcommand{\bU}{{\mathbf U}}
\newcommand{\bW}{{\mathbf W}}
\newcommand{\bV}{{\mathbf V}}
\newcommand{\bZ}{{\mathbf Z}}
\newcommand{\bY}{{\mathbf Y}}
\newcommand{\bF}{{\mathbf F}}
\newcommand{\bP}{{\mathbf P}}

\newcommand{\bl}{\lambda}

\newcommand{\ep}{\epsilon}

\newcommand{\kk}{\kappa}

\newcommand{\mcL}{\mathcal L}

\newcommand{\mcO}{\mathcal O}
\newcommand{\mcP}{\mathcal P}

\newcommand{\FF}{\mathbb{F}}

\newcommand{\bd}{\bar d}

\newcommand{\fp}{\mathfrak p}

\newcommand{\ti}{\tilde}

\newcommand{\emp}{\emptyset}

\newcommand{\sing}{\operatorname{sing}}
\newcommand{\codim}{\operatorname{codim}}

\newcommand{\ms}{\medskip}

\usepackage{xcolor}
\usepackage{amscd}

\begin{document} 
\title[ Applications of Algebraic Combinatorics to Algebraic Geometry]  {Applications of Algebraic Combinatorics to Algebraic Geometry}

\author{David Kazhdan and Tamar Ziegler}
\dedicatory{Dedicated to the memory of Tonny Springer}

\thanks{The second author is supported by ERC grant ErgComNum 682150. }

\maketitle

\begin{abstract} We formulate a number of 
new results 
in Algebraic Geometry and outline their derivation 
from Theorem \ref{uniform} which belongs to
 Algebraic Combinatorics.
 \end{abstract}

 \section{Introduction} In this paper we present a derivation of a number of  results in  Algebraic Geometry from  a theorem (Theorem \ref{uniform}) in Algebraic Combinatorics and it extension   (Theorem \ref{Mainp}) to the case of  $p$-adic rings. In the first half of the article we  outline the derivations  of results in  Algebraic Geometry relying on results of  \cite{kz} and \cite{kz-uniform}. Afterwards we present a proof of 
an extension of Theorem \ref{uniform} to the $p$-adic realm and an application of this result to
Algebraic Geometry.

\subsection{Basic definitions}
We will denote finite fields by $k$ and general fields by $K$.
It will be  important to distinguish between an algebraic variety $\bX$ defined over a field  $K$ and the set $X:=\bX (K)$ of $K$-points of $\bX$. For a polynomial 
$P: V\to K$ we write $\bX _P:= P^{-1}(0)\subset \bV$ and 
$X _P:= P^{-1}(0)\subset V$. So $X _P = \bX _P (K)$.

In this paper we heavily use the notion of rank of polynomials introduced in a paper by  Schmidt \cite{S}, which characterizes complexity of polynomials.
\footnote{This notion of complexity was also  introduced later in the work of Ananyan and Hochster \cite{ah} where it is called {\em strength}.}.

\begin{definition}[Schmidt rank]
 Let $P$ be a polynomial of degree $d$ on a  $K$-vector space $V$. We define the {\em rank} $r_K(P)$ as  the minimal number $r$ such that $P$ can be written in the form $P=\sum _{i=1}^rQ_iR_i$, where $Q_i,R_i$ are $K$-polynomials on $V$ of degrees $<d$. 
\end{definition}

When there is no confusion we will omit the subscript $K$ from $r_K(P)$. 
\begin{example}\label{2e}
If $P:V\to K$ is a non-degenerate  quadratic form then
$\mathrm{dim}(V)/2\leq r(P)\leq 3\mathrm{dim}(V)/2$.
\end{example}

We show that  morphisms $\bP :\bV \to \bA ^c$ defined by polynomials 
of sufficiently high rank and degree less then characteristic of $K$
 possess  a number of nice properties. In particular we show that these morphisms are flat and that their fibers are complete intersections with rational singularities.  In the case of fields of small characteristic we have to replace
 the rank $r(P)$ by the  non-classical rank $r_{nc}(P)$ defined in \ref{nc-rank}.

For polynomials over finite fields there is also  a notion of an analytic rank.
\begin{definition}[Analytic rank]  Let $k=\mF_q$ and 
$P$ be a polynomial of degree $d$ on a  $k$-vector space $V$. Fix a non trivial additive character $\psi:k \to \mC^*$ and define the {\em analytic rank} $a_k(P):= -\log_q |q^{-\dim V}\sum _{v\in V} \psi(P(v))|$. 
\end{definition}

To simplify notations we assume throughout most of the introduction that  $V=V_1\times \dots \times V_d$ and that polynomials
$P:V\to K$ are multilinear. In this case the analytic rank  does not depend on a choice a non-trivial additive character. 

 Most of the results in this paper are based on 
the following Theorem relating rank and analytic rank for multilinear polynomials over finite fields (see \cite{bl, Mi}): 

\begin{fact}\label{biasrank}  There exist a function $\gamma _d(s)$, $\lim _{s\to \infty} \gamma _d(s)=\infty$, such that for any finite field  $k$, and degree $d$ multilinear polynomial $P$ we have  $ a_k(P) \ge \gamma_d(r_k(P)) $. 
\end{fact}

\subsection{Basic conjectures and questions}

Before describing the our results we formulate some basic questions and conjectures regarding the notions of rank defined above. 
\begin{conjecture}[d]\label{d}For any $d\geq 2$ there exists $\kk _d>0$ such that 
for any  multilinear polynomial $P$ we have 
$ r_K (P) \leq \kk _d r_{\bar K}(P)$ where $\bar K$ is the algebraic closure of $K$.
\end{conjecture} 
\begin{remark}\label{3}
\leavevmode
\begin{enumerate}
\item It is easy to see that  Conjecture \ref{d} holds for $d=2$  with $\kk _2=1/2$.
\item Since in the case when $d=3$  the Schmidt rank is equal to the {\it slice} rank the validity of Conjecture \ref{d}(3) follows from Theorem 2.5 and Proposition 4.9 in \cite{D}.
\end{enumerate}
\end{remark}

\begin{conjecture}[d] \label{main}For any $d\geq 2$ there exists $\ep  _d>0$ such that  $a_{\mF _q}(P) \geq \ep _d r_{\mF _q}(P)$
for any  multilinear polynomial of degree $d$.
\end{conjecture} 

\begin{remark}\label{leq}
\leavevmode
 \begin{enumerate}
\item The inequality $a_{\mF _q}(P)\leq r_{\mF _q}(P) $ is known. See \cite{kz-approx, lovett-rank}.
\item It is easy to see that Conjecture \ref{main} holds for $d=2$ with $\ep _2=1/2$.
\end{enumerate}
\end{remark}

Let $K$ be  an algebraically close field, and $(P_s)_{s\in S}$ be a family of polynomials. We define  $U\subset S$ as the subset of points 
$t\in S$ such that $r(P_t) = \max _{s\in S} r(P_s) $. It is easy to see that this  subset is constructible.
\begin{question} 
Is the subset $U \subset S$ locally closed?
\end{question}

\subsection{Main results} 

We formulate in the introduction special cases of the main results. To simplify notations we assume that $K$ is an algebraically closed field of characteristic $>d=\deg P$. 
For the validity of these results we do not have to assume that polynomials $P$ are multilinear.

\begin{theorem} For any $d\geq 2$ there exists $r(d)$ such that varieties $\bX _P\subset \bV$ are geometrically irreducible normal varieties with rational singularities.
\end{theorem}

For a formulation of two other results we need a couple of definitions. 

\begin{definition} A polynomial  $P$ of degree $d$ on a $K$-vector space $K^m$
is {\it m-universal} if for any
$Q\in K[x_1,\dots ,x_m] $  of degrees $\leq d$ there exists an affine map 
$\phi :K^m\to V$ such that $ Q= \phi ^\star ( P)$.
\end{definition}

\begin{theorem}[Universal]\label{universal} There exists a function $R(d,m)$ such that any 
 polynomial $P$
of degree $d$ and rank $ \geq R(d,m) $ is $m$-universal.

\end{theorem}

\begin{definition}
Let   $V$ be a $K$-vector space and $X$ a  subset of $V $. 
A function  $f:X\to K$ is  weakly polynomial of degree $\leq a$, if for any  affine subspace $L\subset X$
the  restriction of $f$ on $L$ is a  polynomial of degree $\leq a$.  
\end{definition}
\begin{theorem}For any $a,d\geq 1$ there exists $r(a,d)$ such that the following holds.

For any polynomial $P:V\to K$ of degree $d$ and rank $\geq r(a,d)$, any weakly polynomial function 
of degree $\leq a$ on $X _P\subset V $ is a restriction of a polynomial of degree $\leq a$ on $V$.

\end{theorem}
The last result we state is for polynomials over finite fields $\mF _q$.
\begin{theorem}
There exists a function  $\rho(d)>0$ such that for any finite field $\mF _q$ of characteristic $>d$ and polynomials $P,Q\in \mF _q[x_1,\dots ,x_N]$ such that $\deg(P)=d$,  $r_{\mF _q}(P)\geq \rho (d)$, $\deg(Q)\leq q/d$ and such that $Q(\bar x)=0$ for all $\bar x\in \mF_q ^N$ such that  $P(\bar x)=0$,  we have $Q=PR$ for some $R\in \mF _q[x_1,\dots ,x_N]$.

\end{theorem}
In the paper all the above mentioned results are extended  to the case when we replace a polynomial $P$ by a 
 collection $\bar P =(P_1,\dots ,P_c)$  of polynomials on $V$.

 \subsection{Acknowledgement}
We thank U. Hrushovski for his help with simplifying the proof of Theorem \ref{main-null}.

\section{The main tool}

We start with defining the Schimdt rank for a family of polynomials. 

\begin{definition}[Schmidt rank]
\begin{enumerate} 
\item Let $P$ be a polynomial of degree $d$ on a  $K$-vector space $V$. We define the {\em rank} $r_K(P)$ as  the minimal number $r$ such that $P$ can be written in the form $P=\sum _{i=1}^rQ_iR_i$, where $Q_i,R_i$ are polynomials on $V$ of degrees $<d$. 
\item For a collection of polynomials $\bar P =(P_i)_{1\leq i\leq c}$  on $V$ we define the rank $r_k(\bar P)$ as the minimal rank of polynomials $P _{\bar a}:= \sum _{i=1}^ca_iP_i$, $\bar a\in K^c \setminus \{0\}$.
\end{enumerate}
\end{definition}

\begin{remark}
 \begin{enumerate} 
\item We will often consider $\bar P$ as a morphism from $\bV$ to $\bA ^c$.
\item When there is  no confusion we write $r(\bar P)$ instead of  $r_k(\bar P)$. 
\end{enumerate}
\end{remark}

When the characteristic of the field is smaller then the maximal degree in $\bar P$ we need to replace the notion of rank with non-classical rank.

\begin{definition}\label{nc-rank}Let $P$ be a polynomial of degree $d$ on a  $K$-vector space $V$.
\begin{enumerate} 
\item  We denote by $\ti P :V^d\to k$ the symmetric multilinear  form given by 
$\ti P(h_1, \ldots, h_d) : =  \Delta_{h_1} \ldots  \Delta_{h_d} P: V^d \to k$, 
where $\Delta_hP(x) := P(x+h)-P(x)$. Note that $\tilde{P}$ only depends on the top-degree part of $P$.
\item  We define the  {\em non-classical rank (nc-rank)} $r_{nc}(P)$ to be the rank of $\ti P$.
\item For a collection $\bar P =(P_i)_{1\leq i\leq c}$ of polynomials on $V$ we define the nc-rank $ r_{nc}(\bar P)$ as the minimal nc-rank of  
polynomials $P _{\bar a}:= \sum _{i=1}^ca_iP_i$, $\bar a\in k^c \setminus \{0\}$.
\end{enumerate}
\end{definition}
\begin{remark}\label{pd} 
\begin{enumerate}
\item If char$(K) >d$ then $r(P) \sim r_{nc}(P)$ (see \cite{kz}).
\item  In low characteristic it can happen that $P$ is of high rank while $\ti P$ is of low rank.
\end{enumerate}
\end{remark}

\begin{example} Let $K$ be a field of characteristic 
$2$, $\bV =\bA ^n$ and  $P(x_1, \dots ,x_n) = \sum_{1 <i<j<k <l\le n} x_ix_jx_kx_l$ is of rank $\sim n$, but of nc-rank $3$, (see \cite{tz}).
\end{example}

 \begin{definition} Let $k:= \FF _q, V$ be a  $k$-vector space and $ \bar P:V\to k^c$ be  a map. We define
\begin{enumerate}
\item  $F _{\bar t} (\bar P) :=\{v\in V|\bar P(v)= \bar t\}$ , $\bar t\in k^c$.
\item  $\nu _{\bar P}:k^c\to \mC$ is the  function given by $\nu _{ \bar P}(\bar t):=|F _{ \bar t}( \bar P)|/q^{\dim(V)-c}$.
\item A map  $\bar P$ is {\it $s$-uniform} if $|\nu _{\bar P}(\bar t)  - 1|\leq q^{-s}$ 
for all $\bar t\in k^c$.
\end{enumerate}
\end{definition}

\begin{lemma}\label{equi} Let $\bar P=(P_1,\ldots ,P_c)$ , $P_i:V\to k$ be a collection of maps such that all the maps $P _{\bar a}:= \sum _{i=1}^ca_iP_i$, $\bar a\in k^c \setminus \{0\}$ are $s$-uniform. Then the map   $\bar P$ is  $(s-c)$-uniform.
\end{lemma}
\begin{proof}
Let $A$ be a finite commutative group written additively and $\Xi$  the group of characters 
$\chi : A\to \mC ^\star$ of $A$. For a function $f: A\to \mC$ we define the Fourier transform $\hat f(\chi) :\Xi \to \mC$ by
$\hat f(\chi)= |A|^{-1}\sum _{a\in A}\chi (a)f(a) $. The next statement is the formula for the inverse Fourier transform. 
\begin{fact}\label{f} 
$\sum _{\chi \in \Xi}\hat f(\chi)\chi (-a_0)=f(a_0) $ for any $a_0\in A$.
\end{fact}
Let $A=k^c$.
We fix a non-trivial additive character $\psi :k\to \mC ^\star$ and associate 
with any $\bar a\in k^c$ a character $\chi _{\bar a}:k^c\to \mC ^\star$ by 
$\chi _{\bar a} (\bar t)=\psi (\langle \bar a,\bar t \rangle)$ where $\langle, \rangle:k^c\times k^c\to k$ is the natural pairing. The map $\bar a\to \chi _{\bar a} $ is a group isomorphism between $A$ and the group $\Xi$ of characters of $A$.

The Fourier transform of the function $\nu$ is given by 
$$\hat \nu_{\bar P} (\chi _{\bar a})=q^{-c}\sum _{\bar t\in k^c}\nu_{\bar P} (\bar t)\chi (\bar t) =
q^{-\dim(V)}\sum _{v\in V}\psi (P_{\bar a}(v))$$
Since, by the assumptions of the Lemma, for any $\bar a\neq 0$ the map $P_{\bar a}$ is  $s$-uniform 
we see that $| \hat \nu_{\bar P} (\chi _{\bar a}) |\leq q^{-s}$ for $\bar a\neq 0$.

As follows from Claim \ref{f} we have
$$\nu _{\bar P}(\bar t)= 
\sum _{\bar a\in k^c}\psi _{\bar a}(-\bar t) \hat \nu _{\bar P} (\bar a)=1+ \sum _{\bar a\in k^c\setminus \{0\}}\psi _{\bar a}(-\bar t) \hat \nu _{\bar P} (\bar a) $$
for all $\bar t\in k^c$. So $| \nu _{\bar P}(\bar t) -1|\leq q^{-(s-c)}$.
\end{proof}

The following result is from Algebraic Combinatorics (see \cite{bl}, \cite{Mi}).

\begin{fact}\label{Main} There exist a function $\alpha _d(s)$, $\lim _{s\to \infty} \alpha _d(s)=\infty$, such that all polynomials $P:V\to k$ of degree $d$ and $r_{nc}(P)\geq \alpha _d(s) $ are 
$s$-uniform. 
\end{fact} 

\begin{remark}\label{rm1}
\begin{enumerate}
\item Fact \ref{Main} is the corner stone of our work.
\item Note that $\alpha _d(s)$  is independent of the finite field $k$ and of the (dimension of the) $k$-vector space $V$.
\item A stronger  form of this result showing the existence of explicitly definable functions $e(d), \gamma(d)$ such that one can take $\alpha _d(s)= e(d)s^{\gamma(d)} $ is proven in \cite{Mi} and \cite{janzer}. This improvement has a number of important applications in Algebraic Combinatorics. One can ask whether it is possible to replace $\gamma(d)$ by $1$.

\end{enumerate}
\end{remark}

\begin{definition}\label{rbar d} For $\bar d=(d_1,\dots , d_c)$ we write 
$r_{\bar d}:= \alpha _{d}(c+2)$.
\end{definition}

The following result,  to which we often refer,   follows immediately from Lemma \ref{equi} and Claim \ref{Main}.

\begin{theorem} [Uniform] \label{uniform} 
Let $\bar P= ( P_1,\ldots ,P_c)$ be a collection of  polynomials of degrees $\le d$ and of $nc$-rank $r\geq r_{\bar d} $. Then 
$| \nu_{\bar {P}}(\bar t) -q^{-c}|\leq q^{-(c+2)}$ for all $\bar t\in k^c$.
\end{theorem}

\section{Irreducibility of fibers}
In this section we show  a derivation of the following result from Theorem  \ref{uniform}.

\begin{theorem}\label{fibers} 
\begin{enumerate}
\item For any field $K$ and a collection $\bar P= (P_1,\ldots ,P_c)$, of polynomials $P_i\in K [x_1,\dots ,x_n] $ of degrees $\bar d=(d_1, \ldots, d_c)$ and nc-rank over the algebraic closure of $K$ is greater then  $ r_{\bar d}$, all the fibers  $\bF _{\bar t}(\bar P):=  \bar \bP ^{-1}(\bar t)\subset \bA ^n$, $\bar t\in K^c$ are non-empty reduced, geometrically   irreducible complete intersections.
\item The schemes $\bF _{\bar t}(\bar P) $ are reduced.
\item  The schemes $\bF _{\bar t}(\bar P) $ are  normal.
\item If $K=\mR$ and all degrees $d_i,1\leq i\leq c$ are odd, then  $\bF _{\bar t}(\bar P)(\mR )\neq \emp ,\bar t\in \mR ^c$.
\end{enumerate} 
\end{theorem}
\begin{remark} For fields of characteristic zero Theorem \ref{fibers} follows from \cite{S}. We include proofs to present our technique in details in  the simplest case.
\end{remark}

\begin{proof} We first prove part $(1)$.
As is well known (see Krull's principal theorem, \cite{M}) any irreducible component $\bY$ of 
$ \bF _{\bar t}(\bar P) $ is of dimension $\geq (n-c)$. So it is sufficient to show that the varieties $ \bF _{\bar t}(\bar P) $ are  irreducible and of dimension $\leq (n-c)$.

We first consider the case when $K$ is a finite field when we can use  the following  result (see \cite{LW}). 

Let $k :=\FF _q$, $k_l :=\FF _{q^l}$. Let $\bX$ be an $m$-dimensional algebraic variety defined over $k$ and  $c(\bX)$ be the number of irreducible components of $\bX$ of dimension $m$ (considered as a variety over the algebraic closure $\bar k$ of $k$). We define $\tau _l(\bX) :=\frac {|\bX (k_l)|}{q^{ml}}$, for $ l\geq 1$.

\begin{claim}\label{W} There exists $u\geq 1$ such that $\lim _{l\to \infty} \tau _{lu}(\bX)=  c(\bX)$.
\end{claim}

For a proof of Theorem \ref{fibers} in the case when $K=\bF _q$ it is sufficient to  observe that  Claim \ref{W} and Theorem \ref{uniform} imply that $\dim(\bF _{\bar t}(\bar P) )=n-c$ and $c( \bF _{\bar t}(\bar P) )=1$.

Now we consider the case when $K$ is an algebraic closure of $\bF _q$.
Since $K=\bigcup \bF _{q^n}$ we may assume that $\bar t\in \bF _{q^n} ^c$. So $ \dim(\bF _{\bar t}(\bar P) )=n-c$ and $c( \bF _{\bar t}(\bar P) )=1$. Therefore  
Theorem \ref{fibers} is proven  in the case when $K$ is an algebraic closure of a finite field.
\bigskip

We start the reduction of the general case of Theorem \ref{fibers} to the case when $K$ is an algebraic closure of a finite field with a reformulation.

\begin{definition} Let $\bar d=(d_1,\dots ,d_c)$ and  $n\geq 1$.
A field $K$ has property $\star (n,\bar d)$ if  for any collection $\bar P = (P_i)_{i=1}^c$ of polynomials $P_i\in K[x_1,\dots ,x_n] $
of degrees $d_i$ and $nc$-rank $\geq r_{\bar d}$, all the varieties $\FF _{\bar t}, \  \bar t\in K^c$ are irreducible of dimension  $\dim(V)-c$, where $r_{\bar d} $ is as in Definition \ref{rbar d}.
\end{definition}

\begin{claim}[n]\label{fiber}
All fields have  the property $\star (n,\bar d)$ for any $ \geq 1, \bar d=(d_1,\dots ,d_c) $.
\end{claim}

\begin{proof}
It is clear that it is sufficient to prove this claim for algebraic closed fields.
Our proof of Claim \ref{fiber}$(n)$  uses the following result from 
Model theory (see \cite{Ma}). 

\begin{fact} Let $T$ be the theory of algebraically closed fields. Then any first order property in $T$ which is  true for all algebraic closures of finite fields 
is true for all algebraically closed fields. 
\end{fact}

Since $\star (n,d)$ is a first order property  in $T$, Claim \ref{fiber}$(n)$, $n\geq 1$ are proven.
\end{proof}

Since the constant $r(\bar d)$ does not depend on $n$, the validity of Claim \ref{fiber}$(n)$ for all $n\geq 1$, implies the validity of the part $(1)$ of
Theorem \ref{fibers}.

\ms

Our proof of parts $(2)$ and $(3)$ is based on the following statement (see \cite{kz-singular}).
\begin{fact}\label{der} There exists a function $\ti r(r,\bar d)$ such that for any collection $\bar P$ of degree $\bar d$ and $nc$-rank $\geq \ti r(r,\bar d) $ there exist $v_i,w_i\in V,1\leq i\leq c$ such that the collection 
$\bar Q= (\bar P, P_i/\partial _{v_i}, \partial P_i/ \partial _{w_i})_{ 1\leq i\leq c }$ is of $nc$-rank $\geq r$. 
\end{fact}

(2) Since (as follows from $(1)$) varieties  $\bF _{\bar t}(\bar P)$ are irreducible complete intersections, to show that schemes  $\bF _{\bar t}(\bar P)$ are reduced  
it is sufficient 
to show that they are reduced at the generic point, and to show the normality of $\bF _{\bar t}(\bar P)$ it is sufficient 
to show that they are non singular in codimension $1$.

To show that a variety $\bF _{\bar t}(\bar P) $ is reduced at the generic point it is sufficient to show the existence 
of vectors $v_1,\dots ,v_c\in V$ such that $ \partial P_i/\partial _{v_i} (x) \not \equiv 0 , x\in \bF _{\bar t}(\bar P) $.

For $\bar v \in V^c$ we define the collection $\bar R(\bar v)$ by  $\bar R(\bar v)= (\bar P, \partial P_i/\partial _{v_i})_{ 1\leq i\leq c} $. Let $\bar d'$ be the degree of the collection $\bar R$. 
As follows from Claim \ref{der}  in the case when $nc$-rank of $\bar P$ is greater than $\ti r(r_{\bar d'}, \bar d) $ 
we can find $\bar v \in V^c$ such the $nc$-rank of $\bar R(\bar v) $ is greater than $r_{\bar d'} $. We fix such $\bar v$ and write $\bar R:= \bar R(\bar v)$.  We consider 
$\bar R$ as a morphism $\bar R: \bV \to \bA^c \times \bA^c$.

As follows from the part (1) we have
 $\dim (\mathbf R _{\bar t,\bar s})= \dim(V)-2c$ for all $\bar t,\bar s \in \bA ^c$.

Let $\bZ \subset \bA ^c$ be the subvariety of $\bar s=(s_1,\dots, s_c)$ such that $s_1s_2\dots s_c=0$ and  
$\bY = \bar R^{-1}(\bar t\times \bZ)\subset \bF _{\bar t}$. Since  $\dim (\mathbf R _{\bar t,\bar s})= \dim(V)-2c$ for all $\bar t,\bar s \in \bA ^c$ we see that $\dim(\bY) = \dim (V)-c-1< \dim (\bF _{\bar t})$. Since the scheme  $\bF _{\bar t} $ is reduced at any $x\in \bF_{\bar t} \setminus \bY$ we see that the scheme $\bF _{\bar t} $ is reduced.

\ms

The proof of $(3)$ is completely analogous. We choose $v_i,w_i\in V$, $1\leq i\leq c$ such that the rank of the collection  
$\bar Q= (\bar P, \partial P_i/\partial _{v_i}, \partial P_i/ \partial _{w_i})_{ 1\leq i\leq c} $ is greater than  $\ti r(r_{\bar d"}, \bar d) $ where $\bar d"$ be the degree of the collection $\bar Q$. We consider $\bar Q$ as a morphism $\bar Q:\bV \to \bA ^c\times \bA ^c\times \bA ^c $. As follows from $(1)$ the fibers of the morphism 
$\bar Q$  are of dimension $\dim(V)-3c$.

Let $\bY ':= \bar Q^{-1}(\bar t \times \bZ \times \bZ)\subset \bF _{\bar t}$. Since  $\dim (\mathbf R _{\bar t,\bar s})=\dim(V)-2c$ for all $\bar t,\bar s \in \bA ^c$ we see that $\dim(\bY) = \dim (V)-c-2= \dim (\bF _{\bar t})-2$. Since the scheme  $\bF _{\bar t} $ is smooth  at any $x\in \bF _{\bar t} \setminus \bY '$ we see that the scheme $\bF _{\bar t} $ is normal.

\ms

For the proof of part $(4)$ we fix a generic subspace $\bW\subset \bV$ of dimension $c$. Since 
$\dim ( \bF _{\bar t}(\bar P))=n-c$ it follows from the B\'ezout theorem that 
$| \bF _{\bar t}(\bar P)\cap \bW(\mC)|= \prod _id_i$ is odd.
Therefore there exists a point $x\in  \bF _{\bar t}(\bar P)\cap \bW(\mC) $ which is invariant under complex conjugation. But then $x\in  \bF _{\bar t}(\bar P)\cap \bW(\mR) $.
\end{proof}

\begin{corollary}\label{real} 
For any infinite field  $K$, $K$-vector spaces $V_i,1\leq  i\leq d $ 
and a collection    $\bar P= (P_1,\ldots ,P_c)$ of polylinear functions $P_i: V_1\times \dots \times V_d\to K$ of degrees $\bar d=(d_1, \ldots, d_c)$ and rank $\geq r_{\bar d}$, we have $\bF _{\bar t}(\bar P)(K)\neq \emp$ for  $\bar t\in K^c$.
\end{corollary}

\begin{proof} We fix   $\bar t\in K^c$ and for
 $w\in \bW:= \bV _2\times \dots \times \bV _d $ define 
$\bZ (w)=\{ v_1\in V_1|(v_1,w)\in \bF _{\bar t}(\bar P) \}$. Let $\bY \subset \bW$ be the subvariety of $w\in \bW $ such that $\dim (\bZ (w))<\dim (V_1)-c$. As follows from 
Theorem \ref{fibers} the closure $\bar \bY \subset \bW$ of  $\bY $ is a proper closed subset. Since $K$ is infinite we have $(\bW \setminus \bY )(K)\neq \emp$. For any $w\in \bW \setminus \bY $ the subset $\bZ (w)$ is defined by a system of linear equations in $V_1$. Since 
 $\bZ (w)\neq \emp$ we see that  $\bZ (w)(K)\neq \emp$. 
\end{proof}


\section{Universality of high rank polynomial mappings}
\begin{definition} A collection $\bar P =(P_i)_{1\leq i\leq c}$ of degree $\bar d=(d_1,\dots ,d_c)$ of polynomials on a $K$-vector space $V$
is {\it m-universal} if for any collection $\bar Q =(Q_i)_{1\leq i\leq c} $ 
of polynomials $Q_i \in K[x_1,\dots ,x_m] $  of degrees $\leq d_i$ there exists an affine map 
$\phi :K^m\to V$ such that $ Q_i= \phi ^\star ( P_i)$ for all $1\leq i\leq c$.
\end{definition}

\begin{theorem}[Universal]\label{universal} There exists a function $R(\bar d,m)$ such that any collection of polynomials $\bar P=(P_1, \ldots, P_c)$  of degrees $\bar d=(d_1, \ldots, d_c)$ and 
$nc$-rank $ \geq R(\bar d,m) $ is $m$-universal, where $K$ is a field
 which is either finite or  algebraically closed.
\end{theorem}

\begin{proof} To simplify notations we only outline the proof in  the case when $c=1$. Let $\Phi$ be the vector space of affine maps $\phi :K^m\to V$   and $L$ be the 
vector space of polynomials $Q\in K[x_1,\dots ,x_m] $ of degree $\leq d$. Choose 
a basis $\lambda _j, j\in J$  of the dual space to $L$. For any polynomial $P$ 
of degree $d$ on $V$ we define the map $ \bar R_P:\Phi \to K^J$  by 
$\bar R_P(\phi )=( \bl _j(\phi ^\star (P)))_{j\in J}$. We have to show the surjectivity of the map $ \bar R_P$.
This surjectivity follows immediately from the following result (see Claim 3.11 in \cite{kz}), and Fact \ref{Main}.

\begin{fact} For any $r$ there exists $h(r,d,m )$ such that the nc-rank of $\bar R_P$ is $\geq r$ for any polynomial $P$ on $V$ of nc-rank $\geq h(r,d,m ) $.
\end{fact}
\end{proof}

\begin{corollary}\label{odd} If $K=\mR$ and all the  degrees $d_i$, $1\leq i\leq c$ in a collection $\bar P$ are odd then under the conditions of Theorem \ref{universal} the collection $\bar P$ is {\it m-universal}.
\end{corollary}
\begin{proof} In this case the polynomials in $\bar R_{\bar P}$ are of the same degrees and Corollary  \ref{odd} follows from the part $(2)$ of Theorem \ref{fibers}.
\end{proof}

\section{The universality for number fields}

Let $K$ be a number field.

\begin{theorem}\label{main0} For any  $m,\bar d=(d_1, \ldots, d_c)$ there exists $r(\bar d,m)$ such that any collection $\bar P: V\to K^c$ of degree $\bar d$ and 
rank $\geq r(\bar d,m) $ is $m$-universal if either $K$ is totally imaginary or all degrees $d_i$ are odd or the collection $\bar P $ consists of polylinear functions.
\end{theorem} 
\begin{proof}Let $\Phi $ be the space of affine maps from $\bA ^m$ to $V$. 
The condition  $\phi ^\star (\bar P)=\bar Q$ is equivalent to a system $\bar G= \bar G(\bar P, \bar Q)$ of $M=\sum _{i=1}^c\binom {m+d_i}{d_i}$ equations of degrees $\leq d=\max_{1 \le i \le c} d_i$ on $\phi \in \Phi$. Let $\bX \subset \Phi$ be the subvariety defined by the system $\bar G$ and $\bX _{\sing}\subset \bX$ be the singular locus.

Let  $D=d^M$, $E=(D-1)D^{ (2^D)}+1$. As follows from Claim 3.10 in \cite{kz} there exist $\ti r$ such 
the inequality  $r_{\mC}( P)\geq r$ implies  $r_{\mC}(\bar G)\geq D!( E+M)$. By  Proposition III$_C$ of \cite{S} we know that $\codim _{\bX}(\bX _{\sing})>E+M$.  Since the field $K$ is infinite 
there exists a $K$- subspace $\bW \subset \Phi$ of codimension $E+M+1$ such that the intersection $\bZ = \bX \cap \bW$ is non-singular and geometrically irreducible. Since  $\dim(\bW)\geq M+E $ and $\codim _{\Phi}(\bX)\leq M$ we have $\dim(\bZ)\geq E$. Since $\deg(\bZ) = \deg(\bX) \leq d^M$ 
we see that $\bZ$ satisfies the conditions of Theorem 4.2 in \cite{w}. Now Theorem \ref{main0} follows from Theorem 1.1 in \cite{bh}, part (2) of Theorem \ref{fibers} and Corollary \ref{real}.
\end{proof}


\section{Weakly polynomial functions}

We start the next topic with a couple of definitions.

\begin{definition}
\begin{enumerate}
\item Let $K$ be a field, $V$ be a $K$-vector space and $X$ a  subset of $V $. 
A function  $f:X\to K$ is {\em weakly polynomial} of degree $\leq a$, if for any  affine subspace $L\subset X$
the  restriction of $f$ on $L$ is a  polynomial of degree $\leq a$.  

\item A subset $X \subset V$ 
satisfies the condition $\star _a$
if any weakly polynomial function of degree $\leq a$ on $X$ is a restriction  of a polynomial function of degree $\leq a$ on $V$.
\end{enumerate}
\end{definition}

The following example demonstrates the  existence of cubic surfaces $X \subset K ^2$  which do not have the property 
$\star_{1}$ for any field $K \neq \FF _2 $.

\begin{example} Let  $V=K^2$, $Q=xy(x-y)$, and let $X$ be the variety defined by $Q$. Then $X=X_0\cup X_1\cup X_2$ where 
$X_0=\{ v\in V|x=0\}, X_1=\{ v\in V|y=0\} , X_2=\{ v\in V|x=y\} $.
The function $f:X\to K$ such that $f(x,0)=f(0,y)=0,f(x,x)=x$ is weakly linear but one can not extend it to a linear function on $V$.
\end{example}

\begin{definition}\label{adm}  A  field $K$ is {\it e-admissible} for $e\geq 1$  if $K^\star$ contains the subgroup $\mu _e$ of roots of order  $e$.
\end{definition}
\begin{remark} To simplify notations we define $C:=\mu _e$.
\end{remark}

\begin{theorem}[Extension]\label{ext} 
There exists an $S=S(a,d)$ such that any hypersurface $Y\subset V$ of degree $d$ and  $nc$-rank $\geq S$  satisfies the condition $\star _a$ if  $K$ is an $ad$-admissible field which is either finite field or is algebraically closed or is 
a number field 
. 
\end{theorem}

\begin{remark}
\begin{enumerate}
\item The main difficulty in a proof of Theorem \ref{ext}  is the non-uniqueness of an extension of $f$ to a  polynomial on $V$ in the case when $a>d$.
\item An analogous statement is true for weakly polynomial functions on subsets  $X_{\bar P}$ where $\bar P$ is a collection of a sufficiently high $nc$-rank.
\end{enumerate}
\end{remark}

\begin{proof} We fix the degree $d$ of $P$.
The proof consists of two steps. We first construct a collection $\bX _m\subset \bV _m$ of hypersurfaces of degree $d$ and nc-rank $\geq  m$ defined over $\mZ$ such that  for all $ad$-admissible fields the subsets $X_m:=\bX_m (K)\subset V_m$ satisfy the condition  $\star _a$ .
In the second step we derive the general case of Theorem \ref{ext}
from this special case.
\subsection{The first step}
We denote by $x^i, 1\leq i\leq d$ the coordinate functions on $\mA ^d$.
\begin{definition} \leavevmode
\begin{enumerate}\item
$\bW:=\bA ^d$ and   $\mu:\bW \to \bA$  is the product $\mu(x^1, \dots ,x^d):= \prod _{j=1}^dx^j$.

\item $ \bV_m:=\bW ^m$ and $Q_m(w_1,\dots ,w_m):= \sum _{i=1}^m \mu(w_i)$,
\item   $\bX _m=\bX _{Q_m}$.
\end{enumerate}
\end{definition}

\begin{proposition}\label{special} \leavevmode
\begin{enumerate}\item $r_{nc}(Q_m)\geq m$.
\item For any $ad$-admissible field $K$ the subvariety $\bX _m (K)\subset \bV _m(K)$ has the property  $\star _a$.
\end{enumerate}
\end{proposition}

\begin{proof} The proof of the inequality $r_{nc}(Q_m)\geq m$ is easy. In the case when $K=\mC$ the result follows immediately from Lemma 16.1 of \cite{S}.

To outline a proof of the second statement we introduce a number of definitions. We fix $m$ and write $X$ instead of $X_m$. Since our field $K$ is $ad$-admissible the group $K^\star$ contains a finite subgroup $C$ isomorphic to $\mZ _{ad}.$
\begin{definition}\label{X} \leavevmode
\begin{enumerate} 
\item $W_1:=\{ (x^1,\dots ,x^d)\in W| x^i=1,i\geq 2\}$.
\item $L:=W_1^m\cap X$. 
\item  $H \subset C^d \subset(K^\star )^d$ is the kernel of the  product map $\mu _C:C^d \to C$. The group $H$ acts on $W$ by $(c^1,\dots ,c^d) (x^1,\dots ,x^d) = (c^1x^1,\dots ,c^dx^d)$.
\item $\Theta$ is the group of characters $\theta :H^m\to K^\star$.
\item We write elements of 
$V_m$ in the form 
$v= (w_1,\dots ,w_m), \ 1\leq i\leq m,\ w_i\in W.$ The group $H^m$ acts on $X_m\subset V_m$ by 
$(h_1,\dots ,h_m) (w_1,\dots ,w_m) = (h_1w_1,\dots ,h_mw_m)$. 
\item $\mcP _a^w(X) \subset k[X] $ is  the subspace of 
 weakly polynomial functions of degree $\leq a$ on $X$.
\item $\mcP _a(X) \subset \mcP _a^w(X) $ is the subspace of functions $f:X\to k$ which are restrictions of polynomial functions on $V$ of degree $\leq a$.
\item For $\theta \in \Theta $, denote by $\mcP _a^w(X)^\theta \subset \mcP _a^w(X) , \mcP _a(X)^\theta \subset \mcP _a(X)$,  the subspaces $\theta$-eigenfunctions. 
\end{enumerate}
\end{definition}
Since $C\subset K^\star$ we have direct sum decompositions 
$\mcP _a^w(X) =\oplus _{ \theta \in \Theta } \mcP _a^w(X)^\theta $ and $\mcP _a(X) =\oplus _{ \theta \in \Theta } \mcP _a(X)^\theta $. Therefore for a proof of Proposition \ref{special} it is sufficient to show the equality 
$\mcP _a^w(X)^\theta = \mcP _a(X)^\theta $ for all $\theta \in \Theta $.

Fix $f\in \mcP _a^w(X)^\theta $. Since  $L\subset V$ is a linear subspace
  the restriction $f|_{L}$ extends to a polynomial on $V$. So (after the subtraction of a polynomial) we may assume that  $f|_{L}\equiv 0$. 
As shown in \cite{kz}  Section $4.2$, any weakly admissible function  $f\in \mcP _a^w(X)^\theta $ vanishing on $L$ is identically $0$. So $f\in \mcP _a(X)^\theta  $.
\end{proof}
\subsection{The second step}

The proof of the general case of Theorem \ref{ext}
is based on the following result. 

\begin{proposition} \label{subspace} There exists a function  $r(d,a)$ such that 
the following holds. Let $K$ be a field which is either finite or 
algebraically closed, $V$ a $K$-vector space, 
$P$ a polynomial of degree $d$ and $W\subset V$ an affine subspace
such that the $nc$-rank of the restriction of $P$ on $W$ is  $\geq r(d,a) $. Then any weakly polynomial function $f$ on $X$ of degree $\leq a$ such that $f _{|X\cap W}$ extends to a polynomial on $W$  of degree $\leq a$  is a restriction of a polynomial of degree $\le a$ on $V$.  
\end{proposition}

\begin{proof} After subtracting a polynomial from $f$ we may assume that  $f| _{X\cap W}\equiv 0$. Using induction on the codimension of $W$ we reduce the Proposition to the case when $W\subset V$ is a hyperplane. We fix a direct sum decomposition $V=W\oplus K$ and denote by $t:V\to K$ the projection. Our proof is by induction on $a$.

The function $g':=f/t$ is defined  on $X\setminus  X\cap W$.
We start with a construction of an extension of  $g'$ to a function $g$ on $X$. Given a point $y\in X\cap W$ consider the set $\mcL$ of lines $L\subset X$ such that  $L\cap W=\{y\}$. Since $f$ is weakly polynomial, the restriction $f_{L}$ is a polynomial $p_L(t)$ vanishing at $0$. We define $g_L(y)$ as the value of $p_L(t)/t$ at $0$.
It is clear that the  following two results  (see Proposition 4.33 of \cite{kz} ) imply the validity of
 Proposition \ref{subspace}.
\begin{claim}Under the conditions of Proposition \ref{subspace},
there exist $g(y)\in K$ such that
\begin{enumerate}
 \item $ | \mcL ^-|/ | \mcL | <1/q$ if $K=\mF _q$ and 
\item $\dim( \mcL ^-)< \dim( \mcL ) $ if $K$ is algebraically closed
\end{enumerate}
where
$ \mcL ^- =\{ L\in \mcL | g_L(y)\neq g(y)\}$. 
\end{claim}
We extend $g'$ to a function on $X$ whose values on $y\in X\cap W$ are equal to $g(y)$.

\begin{claim}
 The function $g:X\to K$ is weakly polynomial of degree $a-1$.
\end{claim}
\end{proof}

Now we can finish the proof of Theorem \ref{ext}. 
Fix $m$ such that $r_{nc}(Q_m) \geq r(d,a) $. Let $R(d,m)$ be as Theorem \ref{universal} or in Theorem \ref{main0} if $K$ is a number field. We claim that  for any admissible field $K$ which is  either finite or algebraically closed, 
any hypersurface $Y \subset V$ of degree $d$ and $nc$-rank $\geq R (d,m) $
satisfies  $\star _a$. Indeed, let 
$f$ be a weakly polynomial function on $X =X_{P}$ of degree $\leq a$ where 
$P:V\to K$ is  a polynomial of degree $d$ of nc-rank $\geq R(d,m) $. 
Since $r_{nc}(P)\geq R(d,m) $
 there exists an affine map $\phi :K^m\to V$ such that $P\circ \phi =Q_m$. It is clear that  the function $f\circ \phi$ 
is  a weakly polynomial function on $X_m$ of degree $\leq a$. Therefore  it  follows from Proposition \ref{special} that the restriction of $f$ on $\text{Im}(\phi)\cap X$ extends to a polynomial on $\text{Im}(\phi)$. It follows now from Proposition 
\ref{subspace} that $f$ extends to a polynomial on $V$. 

\end{proof}

\section{Nullstellensatz}

Let $K$ be a field and $V$ be a finite dimensional $K$-vector space. We denote by $\bV$ the 
corresponding $K$-scheme, and by $\mcP (V)$ the algebra of polynomial functions on $\bV$ defined over $K$.

For  a finite collection  $\bar P = (P_1,\ldots,P_c)$ of polynomials on $\bV$ we denote by $J(\bar P)$ the 
ideal in $ \mcP (V) $ generated by these polynomials, and by $\bX _{\bar P}$ the subscheme of $\bV$ defined by 
this ideal.

Given a polynomial $R \in \mcP(\bV)$, we would like to find out whether it belongs to the ideal 
$J(\bar P)$.  It is clear that the following condition is necessary for the inclusion $R\in J(\bar P)$.

\medskip
(N)   $R(x) = 0$ for all $K$-points $x \in \bX_{\bar P}(K) $.
\medskip

\begin{claim}[Nullstellensatz]\label{Null} Suppose that the field $K$ is algebraically closed  and that the ideal  $J(\bar P)$ is radical.
Then any polynomial $R$ satisfying the  condition $(N)$ lies 
in  $J(\bar P)$. 
\end{claim}

We will show that the analogous result hold for $k=\mF _q$ if $\bX _{\bP}$ is of high $nc$-rank.

From now on we  fix a degree vector $\bd = (d_1,\ldots,d_c)$ and write
 $D:=\prod _{i=1}^cd_i$. We denote by $\mcP _{\bar d}(\bV)$ the space of  $\bar d$-families of polynomials  $\bP = (P_i)_{i=1}^c$ on $V$ such that $\operatorname{deg}(P_i) \leq d_i$. 

\begin{theorem}\label{main-null} There exists  and an bound $r(\bd)>0$ such that for any finite field $k=\mF _q, a< q/D$ and  any collection $\bar P$ of degrees $\bar d$ and nc-rank $\geq r(\bd) $  the following holds. Any polynomial 
$Q$ on $V$ of degree $a$ vanishing on $\bX_{\bar {P}}(k)$ belongs to the ideal 
$J(\bP)$. 
\end{theorem}
\begin{proof}  Our proof is based on  the following {\it rough bound} (see \cite{hr}).
\begin{lemma}\label{rb} Let $\bar P=(P_i)_{i=1}^c \subset \mF _q[x_1,\dots ,x_n]$ be a collection of polynomials of degrees $d_i,1\leq i\leq c$ such that the variety 
$\bY :=\bX _{\bar P} \subset \bA ^n$ is of dimension $n-c$. Then
 $|\bY (\mF _q)|\leq q^{n-c}D$ where $D:=\prod _{i=1}^cd_i$.
\end{lemma}
For the convenience we reproduce the proof of this result.

\begin{proof} Let $F$  be the algebraic closure of $\mF _q$. Then $\bY (\mF _q)$ is the  intersection of $\bY(F)$ with hypersurfaces $\bY _j, 1\leq j\leq n$ defined by the equations $h_j( x_1,\dots ,x_n)=0$ where 
$h_j( x_1,\dots ,x_n)= x_j ^q-x_j$.

Let $H_1,\ldots,H_{n-c}$ be generic linear combinations of the $h_j$ with algebraically independent 
coefficients from an transcendental extension  $F'$ of $F$ and  $\bZ _1,\ldots ,\bZ _{n-c}\subset \bA ^n$ be the corresponding hypersurfaces. 

Intersect successively $\bY$ with $\bZ _1,\bZ _2,\dots ,\bZ _{n-c}$.   Inductively we see that for each $j \leq n-c$, each component $C$ of the intersection
 $\bY \cap  \bZ _1 \cap \dots \cap \bZ _j$ has dimension $n-c-j$.   Indeed,  passing from $j$ to $j+1$ for $j<n-c$ we have $\dim(C)=n-c-j>0$. So  not all the functions $h_i$ vanish on $C$. Hence by the genericity of the choice of linear combinations $\{H_j\}$ we see that 
$H_{j+1}$ does not vanish on $C$ and therefore $\bZ _{j+1}\cap C$ is 
of pure dimension $n-c-j-1$.    Thus  the intersection
 $\bY \cap  \bZ _1 \cap \dots \cap  \bZ  _{n-c}$ has dimension $0$.  This concludes the induction step.

By B\'ezout's theorem we see that  $|\bY \cap  \bZ _1 \cap \dots \cap  \bZ _{n-c}|\leq q^{n-c}D$. Since $\bY (\mF _q)=\bY (\mF _q) \cap  \bZ _1 (\mF _q)\cap \dots \cap \bZ _n (\mF _q)\subset \bX (\mF _q) \cap  \bY _1 (\mF _q)\cap \dots  \cap \bY _{n-c} (\mF _q) $ we see that 
 $|\bY (\mF _q)|
\leq q^{n-c} D$.
\end{proof}

Now we can finish the proof of Theorem \ref{main-null}. Let $R \in \mF _q[x_1,\dots ,x_n] $ be a polynomial of degree $a$ vanishing on the set $\bX _R(\mF _q)$. Suppose that $R$ does not lie in the ideal generated by the $P_i$, $1\leq i\leq c$. Let  $\bY :=\bX _{\bar P}$. Since $R$ vanishes on 
$\bX _{\bar P}(\mF _q)$ we have $\bY (\mF _q) = \bX _{\bar P}(\mF _q) $.

Since $\bX _{\bar P} $ is irreducible we have $\dim(\bY)=n-c-1$. As follows from Lemma \ref{rb} we have the upper bound
$|\bY (\mF _q)| \leq aD q^{n-c-1}$. On the other hand as follows from Theorem \ref{uniform} there exists an  bound $r(\bd)>0$ such that the condition $r_{nc}(\bar P)\geq r(\bd) $  implies the inequality 
$|\bX _{\bar P}(\mF _q)|> q^{n-c}\frac{q-1}{q}$, which is a contradiction to $q>aD$. 
\end{proof}


\section{$p$-adic bias-rank}\label{section-pbr}
For the last two applications we formulate and  prove an analogue of Claim \ref{Main} for local  non-archimedian fields of characteristic zero.

Let $K/\mQ _p$ be a finite extension of degree $u, \mcO \subset K$ the ring of integers, $\fp \subset \mcO$ the maximal ideal and $q=|\mcO /\fp|$. For  $l\geq 1$ we write $A_l= \mcO /\fp ^l$.

\begin{definition} 
\begin{enumerate}
\item For a map $P:A^n_l\to A_l$ and $a\in A_l$ we write 
$\nu (a):= q^{(1-n)l} |P^{-1}(a)| $.
\item A map $P$ is $s$-uniform if $|\nu (a)-1|\leq q^{-s}$.
\item We denote by $\Xi _l$ the group of additive characters $\chi :A_l\to \mC ^\star$. 
\item For $\chi \in \Xi _l$ we  denote  by 
$d(\chi)$ the smallest number $d$ such that $\chi _{|\fp ^dA_l}\equiv 1$.
\item  $\Xi ^d_l\subset \Xi _l$ is the subset of characters $\chi$ such that $d(\chi)=d$.
\item For $\chi \in \Xi$ we  define $b( P;\chi):= q^{-nl}\sum _{v\in A^n_l}\chi (P(v))= q^{-nl} \sum _{a\in A_l}\nu (a)\chi (a)$.
\item For  $P:A^n_l\to A_l$ of degree $d$, we denote $\ti P:(A^n_l)^d\to A_l$ the multilinear form $\ti P(h_1, \ldots h_d) = \Delta_{h_d}\ldots \Delta_{h_1} P$, where $\Delta_{h} P(x)=P(x+h)-P(x)$. 
\item We write $\hat P$ for the reduction of $\ti P$ mod $\fp$.
\end{enumerate}
\end{definition}

\begin{claim} If $ |b( P;\chi) | < q^{-sd(\chi)} $ for all  $\chi \in \Xi$ 
then  $P$ is $(s-2)$-uniform. 
\end{claim}
\begin{proof} The proof is completely analogous to the proof of Lemma \ref{equi}.
\end{proof}

\begin{theorem}\label{Mainp} There exists a function 
 $r(d,s,u)$ such that the following holds: for any  polynomial $P: A_l^n\to A_l $ of degree $d$ such that  the reduction  
$\hat  P: (\mF _q^n)^d\to  \mF _q$ of $\ti P$ is  of rank $\geq r(d,s,u) $,  we have 
$ |b( P;\chi) | < q^{-sd(\chi)}$ 
for all $ \chi \in \Xi _d$.
\end{theorem}

To simplify notations we assume that $K=\mQ _p$. The proof in the general case is completely analogous. We denote   $V_l = A_l^n$, and $e(x)=e^{2 \pi i x}$. We introduce the Gowers uniformity norms: if $G$ is a finite abelian group $f :G \to \mC$ we define
\[
\|f\|^{2^m}_{U_m}  = \mE_{x, h_1, \ldots,h_m\in G }\partial_{h_m} \ldots \partial_{h_1}f(x)
\]
where $\partial_{h}f(x) = f(x+h)\overline{ f(x)}$.  Note that $\partial_{h}\chi(P(x)) = \chi(\Delta_hP(x))$

Our proof of Theorem \ref{Mainp}, which  is by induction in $d,l$ is based on the following stronger result.

\begin{proposition}[d,l]\label{B} For any $s>0$, there exists $r=r^B(d,s)$ such that for any polynomial $S$ of degree $<d$, any $m$, if $\hat R$ is of rank $>r$ then  
$$ \big|\mE_{x \in V^d_l} e \big(\ti  R(x)/\fp^l +S(x)/\fp^m\big) \big| < q^{-sl} .$$
\end{proposition}
\begin{remark} The constant $r=r^B(d,s)$ does not depend on $l$.
\end{remark}
\begin{proof}

We start with the following result. 

\begin{lemma}\label{l-bound} Let  $r= \alpha_d(2^{2d-1}ds)$ be as in Fact \ref{Main}. Then 
Proposition \ref{B}(d,l) holds for $l \le d$. 
\end{lemma}

\begin{proof} 
It suffices to  show that  for any $1\le l\le d$, for any polynomial $R: V_l \to A_l$ of degree $d$ such that the rank of $\hat R$ is  $>r$,  a polynomial $S:V_l^d$ of degree $<d$, and any $m$ we have 
$$\big|\mE_{x \in V^d_l} e\big( \ti R (x)/\fp^l+S(x)/\fp^m\big) \big| <q^{-ds}.$$

Suppose that 
$$\big|\mE_{x \in V^d_l} e\big( \ti R (x)/\fp^l+S(x)/\fp^m\big) \big| \ge q^{-ds}.$$ for some $1<l \le d, m\geq 0$.
We observe that by the CS inequality we have
\[\begin{aligned}
&\big|\mE_{x \in V^d_l} e\big( \ti R (x)/\fp^l+S(x)/\fp^m\big) \big|^2  \le \mE_{y \in V^d_{l-1}} \big| \mE_{x_1 \in V_{l-1}} e\big( \ti R (x_1, y)/\fp^l+S(x_1,y)/\fp^m\big) \big|^2 \\
&= \mE_{y \in V^d_{l-1}} \mE_{x_1,t \in V_{l-1}} e\big( (\ti R (x_1+t, y)-\ti R (x_1, y))/\fp^l+(S(x_1+t,y)- S(x_1,y))/\fp^m\big). \\
\end{aligned}\]
whiih is equal to 
\[
\mE_{y \in V^d_{l-1}} \mE_{x_1,t \in V_{l-1}} e\big( \ti R (t, y)/\fp^l+(S(x_1+t,y)- S(x_1,y))/\fp^m\big)
\]
Applying CS inequality $d-1$ more times in the coordinates $x_2, \ldots, x_d$ we obtain (since $S$ is of degree $<d$):
$$\big(\mE_{x \in V^d_l} e\big( \ti R (x)/\fp^l\big)\big)^{1/2^d}  \ge q^{-ds}.$$
In other words
$$\mE_{x \in V^d_l} e\big( \ti R (x)/\fp^l\big)  \ge q^{-2^d ds}.$$
We rewrite the LHS as
\[
\mE_{x \in V^d_l} e\big( \ti R (x)/\fp^l\big) = \mE_{x} \mE_{y \in V^d_{l-1}}e\big( \ti R (x+\fp y)/\fp^l\big).  
\]
We observe that by the CS inequality we have
\[\begin{aligned}
&| \mE_{x} \mE_{y \in V^d_{l-1}}e\big( \ti R (x+\fp y)/\fp^l\big)|^2   \le   \mE_{x} \mE_{ y' \in  V^{d-1}_{l-1} }| \mE_{y_1 \in V_{l-1}}e\big( \ti R (x_1+\fp y_1, x'+\fp y')/\fp^l\big)|^2 \\
&=  \mE_{x} \mE_{ y' \in  V^{-1}_{l-1} }  \mE_{y_1,h_1 \in V_{l-1}}e\big( (\ti R (x_1+\fp (y_1+h_1), x'+\fp y') -\ti R (x_1+\fp y_1, x'+\fp y')) /\fp^l\big) \\
\end{aligned}\]
Which is equal to 
\[
  \mE_{x} \mE_{ y' \in  V^{-1}_{l-1} }  \mE_{y_1,h_1 \in V_{l-1}}e\big( (\ti R (\fp h_1, x'+\fp y') /\fp^l\big) =  \mE_{x} \mE_{ y' \in  V^{d-1}_{l-1} }  \mE_{y_1,h_1 \in V_{l-1}}e\big( (\ti R (h_1, x'+\fp y') /\fp^{l-1}\big).
\]

Applying the CS inequality $l-2$ more times we see that 
\[
\mE_{x \in V^d_l} e\big( \ti R (x)/\fp^l\big) \le [\mE_{x} \mE_{h_i \in V_{l-1}}e\big(  \ti R(h_1, \ldots, h_{l-1}, x_{l}, \ldots, x_d) /\fp \big)]^{1/2^{l-1}},
\]
and therefore
\[
\mE_{x \in V_1^d}e\big(  \hat R(x_{1}, \ldots, x_d) /\fp \big) \ge q^{-2^{d+l-1}ds}.
\]
Thus $\hat R$ is of rank $\le \alpha_d(2^{d+l-1}ds)$.
\end{proof}


We prove Proposition \ref{B}(d,l) by induction on $d,l$. Let $P$ be  such  that $\hat P $ is of rank $>c(d,s)$ where 
$$c(d,s)=\max \big\{  \alpha_d(2^{2d-1}ds), \alpha_d\big(2^d\big(r^B(d-1, 2s+1)+2ds+1\big)\big) \big\}.$$
By  Lemma \ref{l-bound}, 
Proposition \ref{B}(d,l) holds when 
$l \le d$. So we assume $l>d$. 

Suppose that  Proposition \ref{B}(d',l') holds if either $d'<d$ or $d'=d$ and $l'<l$. If Proposition \ref{B}(d,l) fails then there exists a degree $d$ polynomial $R$ such that  $ \ti R :V^d_l \to A_l$ is of rank $> c(d,s)$
and some polynomial $S:V^d_l \to A_l$ of degree $<d$ and $m \in \mN$ such that:
\[
\big|\mE _{x \in V^d_l} e\big(\ti R(x)/\fp^l+S(x)/\fp^m\big) \big| \ge 1/q^{ls}.
\]

From this it follows that 
\[\begin{aligned}
\frac{1}{q^{2ls}} &\le \big|\mE_{x \in V_l^d} e\big(\ti  R(x)/\fp^l+S(x)/\fp^m\big)\big|^2 \\
&=   \mE_{x,y \in V_l^d} e\big((\ti R(x+y)- \ti R(x))/\fp^l+(S(x+y)-S(x))/\fp^m)\big) \\
&= \mE_{t \in V^d_1} \mE_{y \in V_l^d: y \equiv t (\fp)} \mE_{x \in V_l^d} e\big((\ti  R(x+y)-\ti R(x))/\fp^l+(S(x+y)-S(x))/\fp^m\big) 
 \end{aligned}\]

Fix $t$ and  consider the inner average:
\[\begin{aligned}
a_t &:= \mE_{y \in V_l^d: y \equiv t (\fp)} \mE_{x \in V_l^d} e\big(( \ti R(x+y)-\ti R(x))/\fp^l+ (S(x+y)-S(x))/\fp^m\big)  \\ \end{aligned}\]
 By shifting $x\to x+y'$ and averaging we see that
\[\begin{aligned}
a_t&=   \mE_{x \in V_l^d} \mE_{y, y' \in V_l^d: y \equiv t (\fp),  y' \equiv 0 (\fp)} \\
& \qquad  e\big(( \ti R(x+y'+y)-\ti R(x+y'))/\fp^l + (S(x+y'+y)-S(x+y'))/\fp^m\big) \\
&=   \mE_{x \in V_l^d} \big|\mE_{ y \equiv  t (\fp)}  e\big((\ti  R(x+y))/\fp^l +S(x+y))/\fp^m\big)\big|^2 \
 \end{aligned}\]

\begin{lemma}
 $|a_t|\le \frac{1}{q^{2(l-d)s}}$
\end{lemma}
\begin{proof}
\[\begin{aligned}
&\big|\mE_{ y \equiv  t  (p)}  e\big((\ti  R(x+y))/\fp^l+S(x+y))/\fp^m\big) \big|\\
&=\big| \mE_{y \in V_l^d,  y \equiv 0(\fp)}  e\big((\ti  R(x+t+y))/\fp^l+S(x+t+y))/\fp^m\big) \big|\\
& =\big| \mE_{ y \in V_{l-1}^d}  e\big((\ti R(x+t+\fp y))/\fp^l+S(x+t+\fp y))/\fp^m\big)  \big| \\
& =\big| \mE_{ y \in V_{l-1}^d}  e\big(\frac{\ti  R(y)}{\fp^{l-d}} + \frac{S(y)}{\fp^{m-d}} +\frac{S_{x,t}(\fp y))}{\fp^{m}}+ \frac{T_{x,t}(\fp y))}{\fp^{l}}\big)  \big|
 \end{aligned} \]
(if $d>m$ the second term does not exist) where $T_{x,t}(\fp y)$ is of degree $\le (d-1)$ in $y$.  By the induction hypothesis  the latter is $\le  \frac{1}{q^{(l-d)s}}$.
\end{proof}

\begin{claim} Let $b_t\geq 0$ be  such that  $\mE_{t\in E} b_t\ge  \frac{1}{q^{2ls}}$ and $b_t\le  \frac{1}{q^{2(l-d)s}}$ then for $\frac{1}{2q^{2ds}}|E|$ many $t$ we have that
$b_t\ge  \frac{1}{2q^{2ls}}$.
\end{claim}

\begin{proof}
Let $F$ be the set where $b_t\ge  \frac{1}{2q^{2ls}}$.   We have that 
\[
  \frac{1}{q^{2ls}}|E|   \le \sum _{t \in E} |a_t| \le \sum _{t \in F}   \frac{1}{q^{2(l-d)s}} + \sum_{t \in E-F} \frac{1}{2q^{2ls}}
\]
Rearranging  we get
\[
  |E|   \le |F|(q^{2ds} -1/2)+  |E| \frac{1}{2},
\]
so that $|F| \ge \frac{|E|}{2(q^{2ds} -1/2)}$. 
\end{proof}

Applying this Claim to $b_t:=|a_t|$ we see that for $\ge \frac{1}{2q^{2ds}}|V_1^d|$ many $t \in V_1^d$ we have that 
\[
\big|\mE_{y \in V_l^d: y \equiv t(\fp)} \mE_{x \in V_l^d} e\big(( \ti R(x+y)-\ti R(x))/\fp^l+ (S(x+y)-S(x))/\fp^m \big)\big|  \ge  \frac{1}{2q^{2ls}}\ge  \frac{1}{q^{l(2s+1)}}.
\]
Now $Q_y=\Delta_y  R$ is of degree $<d$ so by the induction on the degree $\hat Q_y$ is of rank $< r^B(d-1, 2s+1)$.

This implies that for $\ge \frac{1}{2q^{2ds}}|V_1^d|$ many $t \in V_1^d$, if $y\equiv t (\fp)$ then  
\[
\big\| \mE_{x \in V_1^d} e\big((\hat R(x+y)- \hat R(x)\big)/p\big) \big\|_{U_{d-1}} \ge  \frac{1}{q^{r^B(d-1, 2s+1)}}
\]
But this now implies that 
\[
\mE_{t \in V_1^d}\big\|\mE_{y \in V_l^d: y \equiv t (\fp)} \mE_{x \in V_l^d} e\big((\hat R(x+y)- \hat R(x))/\fp\big)  \big\|_{U_{d-1}}\ge    \frac{1}{q^{r^B(d-1, 2s+1)}q^{2ds+1}}
\]
But LHS is bounded by  
\[
\| e(\hat R(x)/\fp)\|^2_{U_d}\le \| e(\hat R(x)/\fp)\|_{U_d} \le \big|\mE_{x \in V_1 } e\big(\hat R(x)/\fp \big)\big|^{1/2^d}    
\]
  and we are given that  $\hat R$ is of rank $>\alpha_d\big( 2^d\big(r^B(d-1, 2s+1)+2ds+1\big)\big)$, contradiction. 

\end{proof}


\section{Rational singularities}
\begin{definition}
Let  $\bX$ be a normal irreducible variety over a field of characteristic zero and $a:\ti \bX \to \bX$ a resolution of singularities. We say that $\bX$ 
 has rational singularities if $R ^i a_\star(\mcO _{\ti \bX})=\{0\}$ for $i>0$.
\end{definition}
\begin{remark} This property of $\bX$ does not depend on a choice of a  resolution $a:\ti \bX \to \bX$.
\end{remark} 

\begin{theorem}\label{r} There exists a function  $r(\bar d), \bar d=(d_1,\dots ,d_c)$ such that the following holds.

 Let $\bar P=(P_i)_{i=1}^c \subset \mC [x_1,\dots ,x_n]$ be  a collection of (not necessarily homogeneous) polynomials of degrees $d_i,1\leq i\leq c$ of rank $\geq r(\bar d)$. Then the variety $\bX _{\bar P}$ has rational singularities.
\end{theorem}
\begin{proof}
To simplify the exposition we assume that  $c=1$. Thus  
$\bar P =(P) $ where $P \in \mC [x_1, \dots ,x_n] $ is a homogeneous polynomial of degree $d$.

We first consider the case when $ P \in K [x_1, \dots ,x_n] $ where $K/\mQ$ is a finite extension. 
In this case there exists an infinite  set $S$ of prime ideals in $\mcO _K$ such that 

\begin{enumerate}
\item  for any $\pi  \in S$ the completion ${\mcO _K}_\pi$ of $\mcO _K$ at $\pi$ is isomorphic to $\mZ _p$ where  $p=\text{char}(\mcO _K/\pi)$, 
\item $P\in {\mcO _K}_\pi [x_1, \dots ,x_n]$  and 
\item the reduction $\hat P \in \mF _p[x_1, \dots ,x_n] $ is of rank $r(P)$ over the algebraic closure of $\mF _p$.
\end{enumerate}

We fix $\pi \in S$ such that $p>\deg(P)$. As follows from Theorem A of \cite{an} it is sufficient to establish the inequalities
$|\bX (\mZ /p^m\mZ)-p^{m(n-1)} |\leq p^{m(n-1)-1/2}, m\geq 1$.

We see that for number fields the validity of Theorem \ref{r} in the case when $ P \in K [x_1, \dots ,x_n] $  follows immediately from Theorem \ref{Mainp}.

We show now how to derive the general case of Theorem \ref{r} from 
the case when $ P \in K [x_1, \dots ,x_n] $ where $K/\mQ$ is a finite extension. 

\begin{definition}
 Let  $a:\bX \to \bY$ be a morphism between complex algebraic varieties such that all the fibers  $X_y$ are normal and geometrically irreducible.
We write $Y=\bY (\mC)$ and 
denote by $Y_a\subset Y$ the subset of points $y$ such that the fiber $X_y$ has rational singularities. 
\end{definition}

\begin{lemma}
If a projective morphism   $a$ is defined over $\mQ$ then the subset $Y_a\subset Y$ is also defined over $\mQ$.
\end{lemma}
\begin{proof}
The proof is by induction in $d=\dim(Y)$.  Suppose that the Lemma is known in the case when the dimension of the base is $<d$. Let $t$ be the generic point of $\bY$ and $ \bX_t$ the fiber of $\bX$ over $t$. Fix a resolution $\ti b: \ti \bX _t\to \bX_t$ over the field $k(t)$ of rational functions on $Y$. Then there exists  a non-empty open subset 
$\bU \subset \bY$ such that $b:=\ti b_{| (a\circ \ti b)^{-1}U}$ is a resolution of $X_U:= a^{-1}(U)$. By definition
 $Y_a\cap U= \{ u\in U| R^ib _\star (\mcO _{\ti X})_u=\{0\}\}$ 
for all $i>0$. Since the sheaves $ R ^i b _\star (\mcO _{\ti X}) $ are coherent 
 we see that the subset  $Y_a\cap U$ is defined over $\mQ$.
On the other hand, the inductive assumption implies that $Y_a\cap (Y\setminus U)$ is  defined over $\mQ$. 
\end{proof} 

Now we can finish a proof of Theorem \ref{r} in the general case. 
Consider the trivial fibration   $\hat a:\bA^n\times \bY \to \bY$ 
where $\bY \subset \mathbf P _d^n $ is the constructible subset of  polynomials of degree $d$ on $\bA ^n$ and of rank $\geq r(d)$. Let 
$\bX \subset \bA ^n\times \bY $ be the hypersurface such that $\hat a^{-1}(P)\cap \bX =\bX _P$ and $a:\bX \to \bY$ be the restriction of $\hat a$ onto $\bX$.
 As follows from Theorem \ref{fibers} and Proposition III$_{\mC}$ of \cite{S} all fibers of $a$ are irreducible and normal. For a proof of Theorem \ref{r} we have to show that $Y_a=Y$. 
The validity of  Theorem \ref{r} in the  case when $ P \in K [x_1, \dots ,x_n] $, where $K/\mQ$ is any finite extension, shows that any point of $Y$ defined over a 
finite extension $K$ of $\mQ$ belongs to $Y_a$. Since the subset  $Y_a$ of $Y$ is is  defined over $\mQ$ we see that $Y_a=Y$. 
\end{proof}



\end{document}